\documentclass{amsart}
\usepackage[all]{xy}
\usepackage{amssymb}
 
\calclayout
\newtheorem{dummy}{anything}[section]
\newtheorem{theorem}[dummy]{Theorem}

\newtheorem{lemma}[dummy]{Lemma}

\newtheorem{corollary}[dummy]{Corollary}

\theoremstyle{definition}
\newtheorem{definition}[dummy]{Definition}

\newtheorem{remark}[dummy]{Remark}

\newcommand{\cH}{\mathcal H}

\newcommand{\cS}{\mathcal S}

\newcommand{\cE}{\mathcal E}
\newcommand{\cT}{\mathcal T}

\newcommand{\bZ}{\mathbb Z}

\newcommand{\cy}[1]{\bZ/{#1}}

\newcommand{\bd}{\partial}

\newcommand{\La}{\Lambda}

\newcommand{\Ga}{\Gamma}

\newcommand{\mmatrix}[4]{\left (\vcenter
{\xymatrix@C-2pc@R-2pc{#1&#2\\#3&#4} } \right )}

\DeclareMathOperator{\Hom}{Hom} 
 \DeclareMathOperator{\rank}{rank}
 
\DeclareMathOperator{\Isom}{Isom}

\DeclareMathOperator{\hepta}{Aut}

\DeclareMathOperator{\id}{id}\DeclareMathOperator{\coker}{coker}
\DeclareMathOperator{\Her}{Her} \DeclareMathOperator{\cd}{cd}
  \DeclareMathOperator{\ev}{ev}
\DeclareMathOperator{\pr}{pr}   
\DeclareMathOperator{\rel}{rel}  \DeclareMathOperator{\gd}{gd}

\newcommand{\hept}[1]{\hepta_{\bullet}(#1)}

\newcommand{\heqpt}[1]{\cE_{\bullet} (#1)}

\newcommand{\hM}{\cH (M)}

\newcommand{\quadtypeM}{[\pi, \pi_2, s_M]}
\newcommand{\quadtypecM}{[\pi, \pi_2, c_*[M]]}
\newcommand{\quadtypecMw}{[\pi, \pi_2, c_*[M], w_2]}
\newcommand{\Ospin}{\Omega^{Spin}}

\newcommand{\Bw}{B\langle w_2\rangle}
\newcommand{\Mw}{M\negthinspace\langle w_2\rangle}
\newcommand{\whept}[1]{\hepta_{\bullet}(#1,w_2)}
\newcommand{\whtildeM}{\widetilde \cH (M,w_2)}

\begin{document}
\title[s-cobordism classification of $4$-manifolds]
{s-cobordism classification of $4$-manifolds through the  group of homotopy self-equivalences}
\author[F. Hegenbarth, M. Pamuk, and D. Repov\v{s}]{Friedrich Hegenbarth, Mehmetc\.{i}k Pamuk, and Du\v{s}an Repov\v{s}}
\address{Department of Mathematics
 \newline\indent
University of Milano
 \newline\indent
Milano, Italy} \email{friedrich.hegenbarth{@}unimi.it}

\address{Department of Mathematics
 \newline\indent
Middle East Technical University
 \newline\indent
Ankara 06531, Turkey} \email{mpamuk{@}metu.edu.tr}

\address{Department of Mathematics and Physics
 \newline\indent
University of Ljubljana
 \newline\indent
Ljubljana 1000, Slovenia} \email{dusan.repovs{@}guest.arnes.si}

\begin{abstract}\noindent
The aim of this paper is to give an $s$-cobordism classification of topological $4$-manifolds 
in terms of the standard invariants using the group of homotopy self-equivalences.  Hambleton 
and Kreck constructed a braid to study the group 
of homotopy self-equivalences of $4$-manifolds.  Using this braid together with the modified 
surgery theory of Kreck, we give an $s$-cobordism classification for certain $4$-manifolds 
with fundamental group $\pi$, such that $\cd \pi \leq 2$.
\end{abstract}

\maketitle
\section{Introduction}

The cohomological dimension of a group $G$, denoted $\cd G$, is the projective dimension of $\bZ$ over $\bZ G$. 
In other words, it is the smallest non-negative integer $n$ such that $\bZ$ admits a projective resolution
$P = (P_i)_{i\leq 0}$ of $\bZ$ over $\bZ G$ of length $n$, satisfying $P_i = 0$
for $i > n$.  If there is no such $n$ exists, then we set $\cd G = \infty$.

In this paper we are going to deal with groups whose cohomological dimension  
is less than or equal to $2$.  This class of groups contains the free groups, knot 
groups and one-relator groups whose relator is not a proper power.
Our aim here is to give an $s$-cobordism classification of topological $4$-manifolds with
fundamental group $\pi$ such that  $\cd \pi \leq 2$, in terms of the standard invariants such 
as the fundamental group, characteristic classes and the equivariant intersection form using the 
group of homotopy self-equivalences.

Let $M$ be a closed, connected, oriented, $4$-manifold with a fixed base point $x_0\in M$.  
Throughout the paper, the fundamental  group $\pi_1(M, x_0)$ will be denoted by $\pi$, 
the higher homotopy groups $\pi_i(M, x_0)$ will be denoted by $\pi_i$. Let $ \Lambda = {\bZ}
\left[\pi \right]$ denote the integral group ring of $ \pi$.
The standard involution $\lambda \rightarrow \overline{\lambda}$ on $\Lambda$ is induced 
by the formula
                    $$ 
                         \sum n_g g \rightarrow \sum n_g g^{-1}
                    $$
for $n_g \in {\bZ}$ and $g \in \pi$. All modules considered in this
paper will be right $\La$-modules. 

The first step in the classification of manifolds is the determination
of their homotopy type.  It is a well known result of Milnor \cite{milnor}
and Whitehead \cite{whitehead-49} that a simply connected $4$-dimensional manifold
$M$ is classified up to homotopy equivalence by its integral intersection
form.  In the non-simply connected case,  one has to work with the 
equivariant intersection form $s_M$ where  
                    $$
                         s_M \colon H_2(M; \La) \times H_2(M; \La) \to \La ; \ (a, b) \to s_M(a, b) = a^*(b) \ .
                    $$ 
This is a Hermitian pairing where $a^*\in H^2(M; \La)$ is the Poincar\'e dual of $a$, 
such that $s_M(a, b) = \overline{s_M(b, a)} \in \La$.  This form does not detect
the homotopy type and the missing invariant is the first $k$-invariant $k_M\in H^3(\pi; \pi_2)$, 
see \cite[Remark~4.5]{hk1} for an example.

Hambleton and Kreck \cite{hk1} defined the quadratic $2$-type as the
quadruple $[\pi, \pi_2, k_M, s_M]$ and the group of isometries of the quadratic $2$-type of $M$,
$\Isom [\pi, \pi_2, k_M, s_M]$, consists of all pairs of isomorphisms
                    $$
                        \chi \colon \pi\to \pi \quad \textrm{and} \quad \psi \colon \pi_2\to \pi_2 \ ,
                    $$  
such that $\psi(gx)=\chi(g)\psi(x)$ for all $g\in \pi$ and $x\in \pi_2$, which preserve the $k$-invariant,
$\psi_*(\chi^{-1})^*k_M = k_M$, and the equivariant intersection form,
$s_M(\psi(x), \psi(y))=\chi_*s_M(x, y)$.  It was shown in \cite{hk1} that the quadratic $2$-type detects the homotopy 
type of an oriented $4$-manifold $M$ if $\pi$ is a finite group with $4$-periodic cohomology.

Throughout this paper $H^3(\pi; \pi_2)=0$,  so we have $k_M=0$. For notational ease we will drop it from the notation 
and write $\Isom [\pi, \pi_2, s_M]$ for the group of isometries of the quadratic
$2$-type.

Let $\hept M$ denote the group of homotopy classes of homotopy self-equivalences of $M$, 
preserving both the given orientation on $M$ and the base-point $x_0\in M$.  
To study $\hept M$, Hambleton and Kreck \cite{hk2} established a commutative braid of 
exact sequences, valid for any closed, oriented smooth or topological $4$-manifold.  
To give an $s$-cobordism classification we use the above mentioned braid together with 
the modified surgery theory of Kreck \cite{kreck}.

In section $2$, we briefly review some background material about the modified 
surgery theory and some of the terms of the braid.  Throughout this paper we always 
refer to \cite{hk2}  for the details of the definitions concerning the braid.  In section $3$, 
we are going to further assume that the the following three conditions are satisfied:
                    \begin{itemize}
                          \item[(A1)] The assembly map $A_4 \colon H_4(K(\pi, 1); \mathbb{L}_0(\bZ)) 
                          \to L_4(\bZ[\pi])$ is injective, where $\mathbb{L}_0(\bZ)$ stands for the 
                          connective cover of the periodic surgery spectrum;
                         \item[(A2)] Whitehead group $Wh(\pi)$ is trivial for $\pi$; and
                         \item[(A3)] The surgery obstruction map $\cT(M\times I, \bd)\to L_5(\bZ[\pi])$   is onto, 
                                           where  $M$ is a closed, connected, oriented $4$-manifold with $\pi_1(M)\cong \pi$.                       
                    \end{itemize}
Note that if the Farrell-Jones conjecture \cite{fj} is true for torsion-free groups,
then $\pi$ satisfies all the conditions above.

Now let $u_M\colon M\to K(\pi, 1)$ be a classifying map for the
fundamental group $\pi$. Consider the homotopy fibration
                  $$ 
                        \xymatrix{\widetilde{M}\ar[r]^p&M\ar[r]^(0.4){u_M}&K(\pi,1)}
                  $$ 
 which induces a short exact sequence
                  $$
                        \xymatrix{0\ar[r]&H^2(K(\pi, 1); \cy2)\ar[r]^(0.55){u_M^*}& 
                         H^2(M; \cy2)\ar[r]^{p^*}&H^2(\widetilde{M}; \cy 2)}.
                  $$
Next we recall the following definition given in \cite{hk1.5}.
                  \begin{definition}
                     We say that a manifold $M$ has $w_2$-type (I), (II), or (III) if one
                     of the following holds:
                         \begin{itemize}
                             \item[(I)]$w_2(\widetilde{M})\neq 0$;
                             \item[(II)]$w_2(M)=0$; or
                             \item[(III)]$w_2(M)\neq 0$ and $w_2(\widetilde{M})=0$.
                        \end{itemize}
                  \end{definition}

Using the braid constructed in \cite {hk2} together with the modified surgery theory of
Kreck \cite{kreck}, we show that for topological $4$-manifolds which have
$w_2$-type (I) or (II), with $\cd \pi \leq 2$ and  satisfying (A1), (A2) and (A3),
Kirby-Siebenmann ($ks$) invariant and the quadratic $2$-type give the  $s$-cobordism 
classification.  Our main result is the following:

\begin{theorem}\label{main}
Let $M_1$ and $M_2$ be closed, connected, oriented, topological $4$-manifolds with 
fundamental group $\pi$ such that $\cd \pi \leq 2$ and satisfying properties $(A1)$, $(A2)$ and $(A3)$.
Suppose also that they have the same Kirby-Siebenmann invariant and 
$w_2$-type (I) or (II). Then $M_1$ and $M_2$ are $s$-cobordant if and only if they 
have isometric quadratic $2$-types.
\end{theorem}

Let us finish this introductory section by pointing out the  differences of methods 
used in this paper and the paper by Hambleton, Kreck and Teichner \cite{hkt} which classifies 
closed orientable $4$-manifolds with fundamental groups of geometric dimension $2$ subject to 
the same hypotheses of this paper. 

The geometric dimension of a group $G$, denoted by  $\gd G$, is the minimal dimension of 
a CW model for the classfying space $BG$.  Eilenberg and Ganea\cite{eilenberg-ganea} showed 
that  for any group $G$ we have, $\gd G = \cd G $ for $\cd G > 2$ and if $\cd G=2$ then $\gd G \leq 3$.  
Later Stallings\cite{stallings} and Swan\cite{swan} showed that  $\cd G=1$ if and only if  $\gd G =1$.  
It follows that $\gd G = \cd G$, except  possibly that there may exist 
a group $G$ for which  $\cd G=2$ and  $\gd G = 3$. The statement that  $\cd G$ and 
$\gd G$ are always equal has become known as the Eilenberg-Ganea conjecture 
(see \cite{m.w.davis} for more details and potential counterexamples to Eilenberg-Ganea conjecture).

Although the Eilenberg-Ganea conjecture is still open,   Bestvina and Brady\cite{bestvina-brady} showed  that 
at least one of the Eilenberg-Ganea and Whitehead conjectures has a negative answer, i.e., 
either there exists a group of cohomological dimension  and geometric dimension a counterexample to the 
Eilenberg-Ganea Conjecture or there exists a nonaspherical subcomplex of an aspherical complex a 
counterexample to the Whitehead Conjecture \cite{whitehead-41}.

Therefore, our main result might be a slight generalization of Theorem C of \cite {hkt}.   
Also our line of argument is different:  we first work with the bordism group over the normal $1$-type and then to 
use the braid constructed in \cite{hk2}, 
we work with the normal $2$-type and the $w_2$-type, whereas in \cite{hkt}, the authors work with 
the \textit{reduced} normal $2$-type and a \textit{refinement} of the $w_2$-type.

\section{Background}

The classical surgery theory, developed by Browder, Novikov, Sullivan and Wall in
the $1960$s, is a technique for classifying of high-dimensional
manifolds.  The theory starts with a
normal cobordism $(F, f_1, f_2)\colon (W,N_1,N_2) \to X$ where $f_1$ and $f_2$ are homotopy equivalences
and then asks whether this cobordism is cobordant $\rel \partial$ to an $s$-cobordism.  There is an
obstruction in a group $L_{n+1}(\bZ[\pi_1(X)])$ which vanishes if and only if this is possible.
Later in the $1980$s Matthias Kreck \cite {kreck} generalized this approach:
\begin{definition}\label{type}(\cite{kreck})
Let  $\xi \colon E\to BSO$ be a fibration.
                   \begin{itemize}
                         \item[(i)] A normal $(E, \xi)$ structure $\bar{\nu} \colon N \to E $ of
                          an oriented manifold $N$ in $E$ is a normal $k$-smoothing, if it is
                          a $(k+1)$-equivalence.
                          \item[(ii)] We say that $E$ is $k$-universal if the fibre of the
                          map $E\to BSO$ is connected and its homotopy groups vanish in
                          dimension $\geq k+1$.
                  \end{itemize}
For each oriented manifold $N$, up to fibre homotopy equivalence,
there is a unique $k$-universal fibration $E$ over $BSO$ admitting a
normal $k$-smoothing of $N$.  Thus the fibre homotopy type of the
fibration $E$ over $BSO$ is an invariant of the manifold $N$ and we
call it the \textit{normal $k$-type} of $N$.
\end{definition}
Instead of homotopy equivalences, Kreck started with cobordisms of normal smoothings
$(F, f_1, f_2)\colon  (W,N_1,N_2)\to X $  where
$f_1$ and $f_2$ are only $[\frac{n+1}{2}]$-equivalences. There is an obstruction in a monoid
$l_{n+1}(\bZ[\pi_1(X)])$ which is elementary if and only if that cobordism is cobordant
$\rel\partial$ to an s-cobordism.

Let $M$ be a closed oriented $4$-manifold.  We work with the normal $2$-type of $4$-manifolds.  
That is we need to construct a fibration   $E \to BSO$ whose finer has vanishing homotopy in 
dimensions $\geq k$ and there exists a $3$-equivalence $M \to E$.
Let $B$ denote the $2$-type of $M$ (second stage of
the Postnikov tower for $M$), i.e., there is a commutative diagram
               $$
                  \xymatrix{M \ar[r]^{c} \ar[d]_{u_M}& B \ar[d]^{u_B} \cr
                  B\pi\ar@{=}[r]& B\pi}
              $$ 
Here $u_M$ is unique up to homotopy and a classifying map for the 
universal covering $\widetilde M$ of $M$.
We can attach cells of dimension $\geq 4$ to obtain a CW-complex
structure for $B$ with the following properties:
              \begin{itemize}
                 \item[(i)] The inclusion map $c: M\to B$ induces isomorphisms
                 $\pi_k(M)\to\pi_k(B)$ for $k \leq 2$, and
                 \item[(ii)] $\pi_k(B) = 0$ for $k \geq 3$.
              \end{itemize}
Note that the universal covering space $\widetilde{B}$  of $B$ is
the Eilenberg-MacLane space $K(\pi_2, 2)$, and the inclusion
$\widetilde{M} \to \widetilde{B}$ induces isomorphism on $\pi_2$.

The class $w_2:= w_2(M)\in H^2(M; \bZ)\cong H^2(B; \bZ)$ gives a fibration and we can form the pullback
          $$
              \xymatrix@!C
              {BSpin\ar[r] & \Bw\ar[r]^j\ar[d]_{\xi} & B\ar[d]^{w_2}\\
                BSpin\ar[r] & BSO\ar[r]^(0.45){w}&K(\cy 2,2)}
          $$ 
where $w$ pulls back the second Stiefel-Whitney class for the universal oriented 
vector bundle over $BSO$.  Note that the fibration $\Bw$ over BSO is the normal $2$-type of $M$ 
and if $w_2 = 0$, then $\Bw=B \times BSpin$.

We have a similar pullback diagram for $M$.  Hambleton and Kreck \cite{hk2},
defined a thickening $\whept{M}$ of $\hept{M}$ and then they established a commutative braid
of exact sequences, valid for any closed, oriented smooth or
topological $4$-manifold.

\begin{definition}\label{thickaut}(\cite{hk2})
Let $\whept{M}$ denote the set of equivalence classes of maps
$\widehat{f}\colon M \to \Mw$ such that (i) $f:=j\circ \widehat{f} $
is a base-point and orientation preserving homotopy equivalence, and
(ii) $\xi\circ \widehat{f} =\nu_M$.
\end{definition}

Given two maps $\widehat{f}, \widehat{g}\colon M\to \Mw$ as above,
we define
         $$
            \widehat{f} \bullet \widehat{g}\colon M\to \Mw
         $$
as the unique map from $M$ into the pull-back $\Mw$ defined  by the
pair $f\circ g\colon M \to M$ and $\nu_M\colon M \to BSO$.  It was
proved in \cite{hk2} that $\whept{M}$ is a group under this
operation and there is a short exact sequence of groups
         $$
           \xymatrix{0 \ar[r]&H^1(M; \cy 2) \ar[r]& 
           \whept{M}\ar[r]& \hept{M}\ar[r]& 1 \ .}
         $$

To define an analogous group $\whept{B}$ of self-equivalences, 
we must first state the following lemma.

\begin{lemma}\label{techB}(\cite{hk2})
Given a base-point preserving map $f\colon M\to B$, there is a
unique extension (up to base-point preserving homotopy)
$\phi_f\colon B\to B$ such that $\phi_f\circ c = f$. If $f$ is a
$3$-equivalence then $\phi_f$ is a homotopy equivalence.  Moreover,
if $w_2\circ f = w_2$, then $w_2\circ \phi_f = w_2$.
\end{lemma}

\begin{definition}\label{thickautB}(\cite{hk2})
Let $\whept{B}$ denote the set of equivalence classes of maps
$\widehat{f}\colon M \to \Bw$ such that (i) $f:=j\circ \widehat{f} $
is a base-point preserving $3$-equivalence, and (ii) $\xi\circ
\widehat f =\nu_M$.
\end{definition}

\begin{theorem}[\cite{hk2}]
Let $M$ be a closed, oriented $4$-manifold.  Then there
is a sign-commutative diagram of exact sequences 
       \vskip .4cm
                  $$
                      \begin{matrix}
                           \xymatrix@C-30pt{\Omega_5(\Mw)\ar[dr] \ar@/^2pc/[rr]&& \whtildeM
                           \ar[dr] \ar@/^2pc/[rr] && \whept{B}\ar[dr]^{\beta}  \\
                           & \Omega_5(\Bw) \ar[dr] \ar[ur]  &&
                           \whept{M}  \ar[dr]^{\alpha}\ar[ur] &&\Omega_4(\Bw) \\
                           \pi_1(\heqpt{B,w_2}) \ar[ur] \ar@/_2pc/[rr]&&
                           \widehat\Omega_5(\Bw,\Mw) \ar[ur]^\gamma \ar@/_2pc/[rr]&&
                           \widehat\Omega_4(\Mw)\ar[ur] }
                      \end{matrix}
                  $$
       \vskip .9cm \noindent 
such that the two composites ending in
$\whept{M}$ agree up to inversion, and the other sub-diagrams are
strictly commutative.
\end{theorem}

During the calculation of the terms on the above braid, we will be interested in certain subgroups of $\hept{B}$ 
and $ \whept{B}$.  Before we introduce these subgroups let us define a homomorphism
                            $$
                                     \widehat{j}\colon \whept{B} \to \hept{B} \quad \textrm{by} \quad
                                     \widehat{j}(\widehat{f})=\phi_f
                            $$ 
where \ $\phi_f\colon B\to B$ \ is the unique homotopy equivalence with \ $\phi_f\circ c \simeq f$, and the following
subgroup of $\whept{B}$
                            $$ 
                               \Isom^{\langle w_2\rangle}\quadtypecM:=\{\widehat{f}\in \whept{B}~|~ \phi_f\in
                               \Isom \quadtypecM \}    
                            $$ 
where $\Isom \quadtypecM := \{\phi \in \hept{B} \,| \, \phi_*(c_*[M]) = c_*[M]\}\ $.

\begin{lemma}\label{hat2}
There is a short exact sequence of groups
                            $$
                                 \xymatrix{0\ar[r]&H^1(M;\cy2)\ar[r]& 
                                 \Isom^{\langle w_2\rangle} \quadtypecM\ar[r]^(0.55){\widehat{j}}& 
                                 \Isom \quadtypecM \ar[r]&1}
                            $$                   
\end{lemma}

\begin{proof}
For any $\phi \in \Isom \quadtypecM$, we have an $f\in \hept{M}$ such
that $c\circ f \simeq \phi \circ c$ (this is basically by \cite[Lemma 1.3]{hk1} ). 
We may assume that the pair $(f, \nu_M)$ is an element of
$\whept{M}$ ( \cite[Lemma 3.1] {hk2} ).  The pair $(c\circ f,
\nu_M)$ determines an element $\widehat{f}$ of $\whept{B}$ for which 
$\widehat{j}(\widehat{f})=\phi_f=\phi$. 

Suppose now that $\widehat{f}$, $\widehat{g}\in \Isom^{\langle w_2\rangle}
\quadtypecM$ such that $h:\phi_f\simeq \phi_g$.  We have the
following diagram

                           $$
                                 \xymatrix{& & K(\cy 2,
                                  1)\ar[d]&\cr \bd(M\times I)\ar@{^{(}->}[d]
                                  \ar[rr]^(0.55){\widehat{f}\sqcup \widehat{g}}& &
                                  \Bw\ar[d]^(0.45){(j, \xi)}&\cr M\times I\ar@{-->}[urr]
                                  \ar[rr]_(0.45){(h\circ c\times \id, \nu_M\circ p_1)}& & B\times
                                  BSO \ .}
                           $$
\vskip .2cm
The obstructions to lifting $(h\circ c\times \id, \nu_M\circ
p_1)$ lie in the groups
                  $$
                     H^{i+1}(M\times I, \bd(M\times I); \pi_i(K(\cy 2, 1)))\cong H^i(M;
                     \pi_i(K(\cy 2, 1))) ,
                  $$  
hence the only non-zero obstructions are in $H^1(M; \cy 2)$. 

Let $\widehat{f}\in \whept{M}$, for any $\alpha \in H^1(M; \cy 2)$, 
we will construct a \ $\widehat{g} \in \whept{M}$ with the
property that $f\simeq g$ and the obstruction to $\widehat{f}$ and
$\widehat{g}$ being equivalent is $\alpha$.  Note that different
maps $M\times I \to K(\cy 2, 2)$ relative to the given maps on the
boundary are also classified by $H^1(M; \cy2)$. So we may think
$\alpha \colon M\times I \to K(\cy 2, 2)$ such that $\alpha
|_{M\times \{0\}}$ and $\alpha|_{M\times \{1\}}$ is the constant map
to the base point $\{ \ast \}$ of $K(\cy 2, 2)$.  Consider the
following diagram 
                 $$
                    \xymatrix{ & & \Mw \ar[d]^(0.5){(j, \xi)}&\cr
                    M\times \{0\}\ar@{^{(}->}[d] \ar[rr]^(0.5){(f, \nu_M)}& & M\times
                    BSO\ar[d]^{\rho} & \cr M\times I\ar@{-->}[urr]^{\widehat{\alpha}}
                    \ar[rr]^{\alpha}& & K(\cy2, 2)}
                 $$
\vskip .1cm

The fibration $\rho\colon
M\times BSO\to K(\cy2, 2)=\Omega K(\cy 2, 3)$ is given by $(x, y)
\to w_2(x)-w(y)$, for which the fiber over the base point is by
definition $\Mw$.  By the homotopy lifting property we have
$\widehat{\alpha} \colon M\times I\to M\times BSO$ making the
diagram commutative.  

Let $\widehat{g}:= \widehat{\alpha}|_{M\times
\{1\}}$, then since $w_2(p_1\circ \widehat{g}(x))=w(p_2\circ
\widehat{g}(x))$, where $p_1$ and $p_2$ are projections to the first
and second components respectively, $\widehat{g}$ actually gives us
a map $M \to \Mw$. Observe that $p_1\circ \widehat{\alpha} \colon
M\times I \to M$ is a homotopy between $f$ and $g$.  In order to
lift this homotopy to $\Mw$, we should have $w_2((p_1\circ \widehat{\alpha})(x,
t))=w((p_2\circ \widehat{\alpha})(x, t))$ for all $x\in M$ and $t\in
I$, which is possible if and only if $\alpha$ represents the trivial
map. Hence $\alpha$ is the obstruction to $\widehat{f}$ and
$\widehat{g}$ being equivalent. 
\end{proof}

\begin{lemma}
The kernel of $\beta$, $\ker(\beta):=\beta^{-1}(0)$,
is equal to $\Isom^{\langle w_2\rangle}\quadtypecM$.
\end{lemma}

\begin{proof}
The map $\beta\colon \whept{B} \to \Omega_4(\Bw)$ is defined by
$\beta(\widehat{f})=[M, \widehat{f}\,] - [M, \widehat{c}\,]$.
For the bordism group $\Omega_4(\Bw)$, we use the Atiyah-Hirzebruch spectral 
sequence, whose $E^2$-term is $H_p(M; \Ospin_q(\ast))$. 

The non-zero terms on the $E^2$-page are $H_0(B; \Ospin_4(\ast))\cong \bZ$ in the $(0, 4)$
position, $H_2(B; \cy2)$ in the $(2, 2)$ position, $H_3(B; \cy2)$ in
the $(3, 1)$ position and $H_4(B)$ in the $(4, 0)$ position.  To understand the kernel,  
we use  the projection to $H_4(B)$.

Let $\widehat{f}\in \whept{B}$ and suppose first that $\widehat{f}\in \ker \beta$, then 
$(j\circ \widehat{f})_*[M]=c_*[M]$.  But since $(j\circ \widehat{f})$ is a
$3$-equivalence, there exists $\phi \in \hept{B}$ with $\phi \circ
c= j\circ \widehat{f}$  (recall Lemma \ref{techB}). So, $\phi_*(c_*[M])=c_*[M]$ which means
$\widehat{j}(\widehat{f})=\phi \in \Isom \quadtypecM $. Therefore
$\ker(\beta)\subseteq \Isom^{\langle w_2\rangle} \quadtypecM$.  To
see the other inclusion note that 
            $$
                 \coker(d_2\colon H_4(B; \cy2)\to
                 H_2(B; \cy2)) \cong \langle w_2\rangle
            $$ 
and the class $w_2$ is preserved by a self-homotopy equivalence.
\end{proof}

\begin{definition} (\cite{hk2})
Let $\whtildeM$ denote the bordism groups of
pairs $(W, \widehat F)$, where $W$ is a compact,
oriented $5$-manifold with $\bd_1 W = -M$, $\bd_2 W = M$ and
the map $\widehat F \colon W \to \Mw$ restricts to $\widehat{id}_M$
on $\bd_1 W$, and on $\bd_2 W$ to a map $\hat f\colon M \to \Mw$ satisfying
properties (i) and (ii) of Definition \ref{thickaut} .
\end{definition}

\begin{corollary}
The images of $\whept{M}$ or $\whtildeM$ in
$\whept{B}$ are precisely equal to $\Isom^{\langle w_2\rangle}
\quadtypecM $.
\end{corollary}

\begin{proof}
Let $\widehat{f}\in \whept{M}$ and $\phi_{\widehat{f}}$ denote the image of $\widehat{f}$ 
in $\whept{B}$.  Then $\widehat{j}(\phi_{\widehat{f}})= \phi_f$ satisfies $\phi_f\circ c=c\circ f$ 
and $\phi_f$ preserves $c_*[M]$.  Hence $\phi_f \in \Isom \quadtypecM $.  Now suppose that 
$\phi \in \Isom \quadtypecM$, then by \cite[Lemma 1.3]{hk1} there exists $f\in \hept{M}$ such that
$\phi \circ f \simeq c\circ f$.  We may assume that $\widehat{f}=(f, \nu_M) \in \whept{M}$ \cite[Lemma 3.1]{hk2}.  Let
$\phi_{\widehat{f}}\in \whept{B}$ denote  the image of $\widehat{f}$, we have $\widehat{j}(\phi_{\widehat{f}})= \phi$.

The result about the image of $\whtildeM$ follows from the exactness of the braid \cite[Lemma 2.7]{hk2} 
and the fact that $\ker(\beta)=\Isom^{\langle w_2\rangle}\quadtypecM$.
\end{proof}

\begin{remark}\label{induced}
By universal coefficient spectral sequence, we have an exact
sequence 
                   $$
                        \xymatrix{0\ar[r]& H^2(\pi; \Lambda)\ar[r]&H^2(M; \La)
                        \ar[r]^(0.45){\ev}&\Hom_{\Lambda}(\pi_2, \Lambda )\ar[r]&0}
                   $$ 
and the cohomology intersection pairing is defined by $s_M(u, v)=
ev(v)(PD(u))$ for all $u, v\in H^2(M; \La)$ where $PD$ is the
Poincar$\acute{e}$ duality isomorphism.  Since $s_M(u, v)=0$ for all
$u\in H^2(M; \La)$ and $v\in H^2(\pi; \Lambda)$, the pairing $s_M$
induces a nonsingular pairing
                   $$
                        s'_M\colon H^2(M; \La)/H^2(\pi; \Lambda)
                        \times H^2(M; \La)/H^2(\pi; \Lambda)\to \La \ . 
                   $$
\end{remark}

Before we finish this section, let us point out that for our purposes we need to look for 
a relation between the image of the fundamental class $c_*[M]\in H_4(B)$ and the equivariant
intersection pairing $s_M$.  Let $\Her(H^2(B; \La))$ be
the group of Hermitian pairings on $H^2(B; \La)$.  We can define a natural
map $F\colon H_4(B)\to Her(H^2(B; \La))$ by
             $$
                  F(x)(u, v) = u(x \cap v)=(u\cup v)(x) \ .
             $$ 
The construction of $F$ applied to $M$ yields $s_M$ and by naturality $F(c_*[M])=s_M$.
In other words, we have the following commutative diagram

              $$
                     \xymatrix{H^2(B; \La)\times H^2(B; \La) \ar[r]^(0.6){F(c_*[M])}
                     \ar[d]^{\cong}_{c^*\times c^*}& \La \cr H^2(M; \La)\times H^2(M;
                     \La)\ar@{=}[r] & H^2(M; \La)\times H^2(M; \La)\ar[u]_{s_M} \ .}
              $$

Therefore any automorphism of $B$ which preserves $c_*[M]$, also preserves the intersection
form $s_M$.  The converse of this statement is not necessarily true, i.e.,  $c_*[M]$ and
$s_M$ do not always uniquely determine each other.

\section{s-Cobordism}

In this section we are going to prove Theorem \ref{main}.  Let $M$ be a closed, connected, oriented, 
topological $4$-manifold with fundamental group $\pi$ such that $\cd \pi \leq 2$.  We study bordism 
classes of such manifolds over the normal $1$-type.  

For type (I) manifolds, $w_2(\widetilde{M})\neq 0$, oriented topological bordism 
group over the normal $1$-type is 
                  $$
                         \Omega_4^{STOP}(K(\pi, 1))\cong
                         \Omega_4^{STOP}(\ast)\cong \bZ\oplus \cy2
                  $$ 
via the signature, $\sigma(M)$, and the $ks$-invariant.  
Recall that $\sigma(M)$ is determined
via the integer valued intersection form $s_M^{\bZ}$ on $H_2(M)$.  Since the image 
                  $$
                        \xymatrix{H^2(\pi; \bZ)\ar[r]^(0.45){u_M}&H^2(M; \bZ)}
                  $$ 
is the radical of $s_M^{\bZ}$
$\sigma(M)$ is equal to the signature of the form $s_M\otimes_{\La}\bZ$  \cite[Remark 4.2]{hkt}.
Therefore when $\cd \pi \leq 2$, the signature of $M$ is determined by the formula 
                 $$
                      \sigma(M)=\sigma(s_M^{\bZ})=\sigma(s_M\otimes_{\La}\bZ) \ .
                 $$
On the other hand, in the type (II)  case, $w_2(\widetilde{M}) = 0$, we have 
                 $$
                      \Omega_4^{TOPSPIN}(K(\pi, 1))\cong \bZ \oplus H_2(\pi; \cy2) \ .
                 $$ 
In this case, the invariants are signature and an invariant in $H_2(\pi; \cy2)$.  

Now, let $M_1$ and $M_2$ be closed, connected, oriented, topological $4$-manifolds with isomorphic 
fundamental groups.  By fixing an isomorphism, we identify $\pi=\pi_1(M_1)=\pi_1(M_2)$.  Suppose also that 
$\cd \pi \leq 2$.  Suppose further that $M_1$ and $M_2$ have isometric quadratic $2$-types.  First we are going 
to show that $M_1$ and $M_2$ are homotopy equivalent  by using  \cite[Corollary 3.2]{baues-bleile}.  
Then we are going to show that they are indeed bordant over the normal $1$-type, if we further assume that $\pi$ 
satisfies (A1).    

Since $M_1$ and $M_2$ have isometric quadratic $2$-types, we have
                  $$
                     \chi \colon \pi_1(M_1)\to \pi_1(M_2) \quad \textrm{and} \quad 
                     \psi \colon \pi_2(M_1)\to \pi_2(M_2)
                  $$
a pair of isomorphisms such that $\psi(gx)=\chi(g)\psi(x)$ for all $g\in\pi$, $x\in\pi_2(M_1)$ 
and preserving the intersection form i.e.,
                  $$
                     s_{M_2}(\psi(x), \psi(y)) = \chi_*(s_{M_1}(x, y)) \ .
                  $$

Let $B(M_i)$ denote the $2$-type of $M_i$ and $c_i\colon M_i \to B(M_i)$ corresponding $3$-equivalences for $i=1, 2$.  
We are going to construct a homotopy equivalence between $B(M_1)$ and $B(M_2)$.  Note that, we have isomorphisms 
$\xymatrix{\pi_2(c_i)\colon \pi_2(M_i)\ar[r]^(0.52){\cong}& \pi_2(B(M_i))}$ for $i = 1, 2$.  
Start with the composition
                 $$
                     \pi_2(c_2)\circ \psi \circ \pi_2(c_1)^{-1}\colon
                     \pi_2(B(M_1))\xrightarrow{\cong} \pi_2(B(M_2)) \ .
                 $$
We can think of any Abelian group $G$ as a topological group with discrete topology.
Then we can define $K(G, 1)=BG$, which is also an Abelian
topological group, and $K(G, 2)=BK(G, 1)=B^2G$.  This construction
is functorial.  Hence we have a homotopy equivalence
                $$
                    B^2(\pi_2(c_2)\circ \psi \circ \pi_2(c_1)^{-1})\colon
                    K(\pi_2(B(M_2)), 2) \to K(\pi_2(B(M_2)), 2)
                $$ 
which is $\pi_1$-equivariant, since $\psi$ is $\pi_1$-equivariant.  We also
have another $\pi_1$-equivariant homotopy equivalence, namely $E\chi
\colon E\pi_1(M_1) \to E\pi_1(M_2)$, where the contractible space
$E\pi_1(M_i)$ is the total space of the universal bundle over
$B\pi_1(M_i)$ for $i=1, 2$. Let 
               $$
                   \tau:= E(\chi)\times B^2(\pi_2(c_2)\circ \psi \circ \pi_2(c_1)^{-1})
               $$ 
and recall that $B(M_i)\simeq E\pi_1(M_i)\times_{\pi_1(M_i)} K(\pi_2(B(M_i)), 2)$.
Then we have 
               $$
                   \tau \colon B(M_1)\to B(M_2) \ .
               $$ 
Also since $B(M_i)$ is a fibration over $B\pi_1(M_i)$ with fiber
$K(\pi_2(B(M_i)), 2)$ by five lemma, we can see that $\tau$ is a
homotopy equivalence.  Summarizing we have a homotopy equivalence
$\tau$ with the following commutative diagram:
\vskip .1cm
               $$
                  \xymatrix{\pi_2(M_1)\ar[rr]^(0.45){\pi_2(c_1)}\ar[d]_{\psi}& & \pi_2(B(M_1))
                  \ar[d]^{\pi_2(\tau)}\cr \pi_2(M_2)\ar[rr]_(0.45){\pi_2(c_2)}& &
                  \pi_2(B(M_2)) \ .}
               $$
Note that we have $\tau_{\sharp}(s_{M_2})=s_{M_1}$.  Since $M_1$ and $M_2$ have isometric 
quadratic $2$-types, they have isomorphic intersection forms, which implies that 
$\tau_*((c_1)_*[M_1])=(c_2)_*[M_2]$  (we may need to use the image of $(c_2)_*[M_2]$ under a 
self-equivalence of $B(M_2)$ if necessary, see \cite[Lemma 3]{hillman} and the proof of \cite[Theorem 14]{hillman}).  
Also see the  discussion at the end of Section $2$ for the relation between the image of the fundamental class and 
the equivariant  intersection form.   Therefore $M_1$ and $M_2$ have isomorphic fundamental triples 
in the sense of \cite{baues-bleile} and hence 
they must be homotopy equivalent by  \cite[Corollary 3.2]{baues-bleile}. 

If we further assume that the assembly map
          \begin{itemize}
            \item[(A1)] $A_4 \colon H_4(K(\pi, 1); \mathbb{L}_0(\bZ)) \to L_4(\bZ \pi)$  \ is injective,
          \end{itemize}
then by \cite[Corollary 3.11]{j.f.davis} $M_1$ and $M_2$ are bordant over the normal $1$-type.


Therefore, if the fundamental group $\pi$ satisfies (A1),
then we have a cobordism $W$ between $M_1$ and $M_2$ over the normal $1$-type, 
which is a spin cobordism in the type (II) case.

Choose a handle decomposition of $W$. Since $W$ is connected, we can cancel all $0$- and $5$-handles. 
Further, we may assume by low-dimensional surgery that the inclusion map $M_1 \hookrightarrow W$ is a 
$2$ equivalence. So we can trade all $1$-handles for $3$-handles, and upside-down, all $4$-handles for 
$2$-handles. We end up with a handle decomposition of $W$ that only contains $2$- and $3$-handles, and view $W$ as
               $$
                   W = M_1\times [0, 1] \cup \{ 2-handles \} \cup 
                   \{ 3-handles \} \cup M_2\times [-1, 0] \ .
               $$
Let $W_{3/2}$ be the ascending cobordism that contains just $M_1$ and all $2$-handles and let  $M_{3/2}$ 
be its $4$-dimensional upper boundary.  The inclusion map $M_1 \hookrightarrow W$ is a $2$ equivalence, 
so attaching map $S^1\times D^3\to M_1$  of a $2$-handle must be null-homotopic.  Hence attaching a 
$2$-handle is the same as connect summing with $S^2\times S^2$ or the same as connect summing with 
$S^2\widetilde{\times} S^2$.  Since $M_1$ and $\widetilde{M_1}$ are spin at the same time, we can 
assume that there are no $S^2\widetilde{\times} S^2$-terms present in $M_{3/2}$ (see for example \cite[p. 80]{pamuk}).

From the lower half of $W$, we have $M_{3/2}\approx M_1\sharp m_1 (S^2\times S^2)$, while from the upper half, 
we have $M_{3/2} \approx M_2\sharp m_2(S^2\times S^2)$. Since $\rank(H_2(M_1))=\rank(H_2(M_2)))$, 
it follows that $m = m_1 = m_2$.  We have a homeomorphism 
                          $$
                               \xymatrix{\zeta \colon M_2 \sharp m(S^2 \times S^2)\ar[r]^{\approx}& 
                               M_1\sharp m(S^2 \times S^2)}.
                          $$

Next assume that:
\begin{itemize}
\item[(A2)] Whitehead group $Wh(\pi)$ is trivial for $\pi$.
\end{itemize}
Hence being $s$-cobordant is equivalent to being $h$-cobordant.  The strategy for the remainder of the proof is the following:
We will cut $W$ into two halves, then glue them back after sticking in an $h$-cobordism of $M_{3/2}$.  This cut and reglue procedure
will create a new cobordism from $M_1$ to $M_2$. If we choose the correct $h$-cobordism, then the $3$-handles from the upper half will
cancel the $2$-handles from the lower half. This means that the newly created cobordism between $M_1$ and $M_2$ will have no
homology relative to its boundaries, and so it will indeed be an $h$-cobordism from $M_1$ to $M_2$.

Note that  we have $\tau_{\sharp}(s_{M_2})=s_{M_1}$ and $s_{M_1}\cong s_{M_2}$ 
if and only if $s'_{M_1}\cong s'_{M_2}$.  Hence we can immediately deduce that
$\tau_{\sharp}s'_{M_1}=s'_{M_2}$.
Now let $M:=M_1\sharp m(S^2 \times S^2)$ and $M':= M_2 \sharp m(S^2 \times S^2)$ 
with the following quadratic $2$-types,
              $$
                  [\pi, \pi_2, s_M]:=[\pi_1(M_1), \pi_2(M_1)\oplus \La^{2m}, s_{M_1}\oplus H(\La^{m}) ]
              $$ 
and
              $$
                  [\pi_1(M_2), \pi_2(M_2)\oplus \La^{2m}, s_{M_2}\oplus H(\La^{m}) ]\ ,
              $$ 
where $H(\La^{m})$ is the hyperbolic form on $\La^{m} \oplus (\La^{m})^*$.  

Since $W$ is a cobordism over the normal $1$-type, 
             $$
                   (\pi_1(\zeta)\circ \chi, \pi_2(\zeta)\circ(\psi\oplus \id))=(\id,
                   \pi_2(\zeta)\circ(\psi\oplus \id))
             $$ 
is an element in $\Isom \quadtypeM$.  Let $B=B(M)$ denote the $2$-type of $M$. We have an
exact sequence of the form \cite{moller}
            \begin{equation}\label{eqn:moller}
                  \xymatrix{0\ar[r]&H^2(\pi; \pi_2)\ar[r]& \hept{B}\ar[r]^(0.45){(\pi_1, \pi_2)}& \Isom [\pi, \pi_2]\ar[r] &1}.
            \end{equation}
Therefore we can find a $\phi'' \in \hept{B}$ such that
            $$
                 \pi_1(\phi'')= \id \quad \textrm{and} \quad \pi_2(\phi'')=\pi_2(\zeta)\circ(\psi\oplus \id).
            $$
The homotopy self-equivalence $\phi''$ preserves the intersection form
$s_M$ but on the braid we see $\Isom^{\langle w_2\rangle} \quadtypecM$.  So to use
the braid, we need to construct a self homotopy equivalence of $B$ which preserves $c_*[M]$.

Hillman \cite{hillman} showed that for $\cd\pi \leq 2$, we have $\pi_2(M)\cong P\oplus H^2(\pi; \La)$ 
where $P$ is a projective $\La$-module.  He also showed that there exists a $2$-connected degree-$1$ 
map $g_{M}\colon M\to Z$ where $Z$ is a $PD_4$ complex with $\pi_2(Z)\cong H^2(\pi; \La)$ and 
$\ker(\pi_2(g_{M}))=P$.  He called $Z$ as the strongly minimal model for $M$. 

We may assume that $\pi_2(g_{M})$ is projection to the second factor and 
$c_Z \circ g_M = g \circ c$ for some $2$-connected map $g\colon B\to B(Z)$, 
where $B(Z)$ denotes the $2$-type of $Z$ . The map $g$ is a fibration with fibre 
$K(P, 2)$, and the inclusion of $H^2(\pi; \La)$ into $\pi_2(M_2)$ determines a section $s$ for $g$. 
Summarizing we have the diagram below with a commutative square

           \[
              \xymatrix{& M\ar[r]^{g_{M}} \ar[d]_{c}&Z \ar[d]^{c_Z}& 
              \cr K(P, 2) \ar[r]& B\ar[r]^(0.4)g& B(Z) \ar@/^1.5pc/@{->}[l]_{s}}
           \] 
Note that since $\phi''$ preserves the intersection form and identity on $\pi$, 
             $$
                 \pi_2(\phi'') \colon P\oplus H^2(\pi; \La) \to  P\oplus H^2(\pi; \La)
             $$  
has a matrix representation of the form
             $$
                 \pi_2(\phi'')=
                 \left[ {\begin{array}{cc}
                 \ast & \ast \\
                 0 & \id \\
                 \end{array} } \right]
              $$
where the first $\ast$ represents an $\pi$-module isomorphism 
$P \to P$ and the second $\ast$ represents an $\La$-module homomorphism 
$P \to H^2(\pi; \La)$.   We modify $\phi''$, first to $\phi'  \in \hept{B}$ 
so that $\pi_2(\phi')$ has a matrix representation of the form
  \[
         \pi_2(\phi')=
            \left[ {\begin{array}{cc}
             \ast & 0 \\
             0 & \id \\
                \end{array} } \right]
   \]              
i.e., it induces the zero homomorphism from $P$ to $H^2(\pi; \La)$.  To achieve this    
first define 
       $$
           \theta \colon P \to H^2(\pi; \La) \qquad \textrm{by} \qquad \theta(p)= \pr_2(\pi_2(\phi'')(p, 0)) \ .
       $$
Then define
       $$
           \alpha_{\theta} \colon P \oplus H^2(\pi; \La) \to P \oplus H^2(\pi; \La) \qquad 
           \textrm{by} \qquad \alpha_{\theta}(p, e)= (p, e-\theta(p)) \ . 
       $$
This newly defined map $\alpha_{\theta}$ is a $\La$-module isomorphism of $\pi_2$ 
by \cite[Lemma 3]{hillman}.  Now the pair $(\id, \alpha_{\theta})$ gives us an isomorphism 
$\phi''_{\theta}$ of $B$ by the sequence ~(\ref{eqn:moller}) on the previous page.  Define   
$\phi' := \phi''_{\theta} \circ \phi''$, and observe that $g\circ \phi' = g$.

Let $L:= L_{\pi}(P, 2)$ be the space with algebraic $2$-type $[\pi, P, 0]$ and universal 
covering space $\widetilde{L}\simeq K(P, 2)$.
We may construct $L$ by adjoining $3$-cells to $M$ to kill the kernel of the projection from 
$\pi_2$ to $P$ and then adjoining higher dimensional cells to kill the higher homotopy groups.  
The splitting $\pi_2\cong P\oplus H^2(\pi; \La)$ also determines a projection $q\colon B\to L$.

To begin with we have the following isomorphisms where $\Ga$ denotes the Whitehead quadratic 
functor \cite{whitehead-50}.

             \begin{align*}
                H_4(B)&\cong \Ga(\pi_2)\otimes_{\La}\bZ \oplus H_2(\pi; \pi_2)\\
                &\cong \Ga(H^2(\pi; \La) \oplus P)\otimes_{\La}\bZ \oplus H_2(\pi; H^2(\pi, \La))\\
                &\cong(\Ga(H^2(\pi, \La))\oplus \Ga(P) \oplus H^2(\pi, \La) \otimes P) \otimes_{\La}\bZ
                \oplus H_2(\pi; H^2(\pi, \La))\\  &\cong \Ga(P)\otimes_{\La}\bZ \oplus
                \Ga(H^2(\pi, \La))\otimes_{\La}\bZ\oplus H_2(\pi; H^2(\pi, \La)) \oplus(H^2(\pi, \La) \otimes P)
                \otimes_{\La}\bZ\\ &\cong H_4(L)\oplus H_4(B(Z))\oplus (H^2(\pi, \La) \otimes P)
                \otimes_{\La}\bZ \ .
             \end{align*}
\vskip.1cm

We are going to consider the difference $\phi'_*(c_*[M])- c_*[M] \in H_4(B)$.  
We start by projecting $\phi'_*(c_*[M])$ and $c_*[M]$ to
$H_4(L)\cong\Ga(P)\otimes_{\La}\bZ$. Recall that we have a nonsingular pairing
            $$
               s'_M\colon H^2(M; \La)/H^2(\pi; \Lambda)
               \times H^2(M; \La)/H^2(\pi; \Lambda)\to \La \ .
            $$ 
If we further restrict $s'_M$ to $\Hom_{\Lambda}(P, \Lambda)\cong H^2(L;
\La)/H^2(\pi; \Lambda)$, we get a Hermitian pairing $s''_M \in
\Her(P)$.  Therefore, we have the following commutative diagram
             $$
                \xymatrix{H_4(B)\ar[d]_{q_*} \ar[rr]^(0.4)F  & &
                \Her(H^2(B; \La))\ar[d]^{q_{\sharp}} \\
                \Ga(P)\otimes_{\La}\bZ \ar[rr]^{\cong}& & \Her(P) \ .}
             $$
The bottom row is an isomorphism  \cite[Theorem 2]{hillman}.  Both
$q_*(c_*[M])$ and $q_*(\phi'_*(c_*[M]))$ map to $s''_M$, hence
$q_*(c_*[M])=q_*(\phi'_*(c_*[M]))$.  Since $g\circ \phi' = g$, we have
             $$
                \phi'_*(c_*[M])-c_*[M] \in (H^2(\pi; \La)\otimes P)\otimes_{\La}\bZ \ .
             $$
As a final modification, as in \cite[Lemma 3]{hillman}, we can choose a self equivalence  
$\phi'_{\theta}$  of $B$ so that  $(\phi'_{\theta} \circ \phi')_*(c_*[M]) = c_*[M] $  mod 
$\Ga(H^2(\pi, \La))\otimes_{\La}\bZ$.   Hence  $(\phi'_{\theta} \circ \phi')_*(c_*[M]) = c_*[M] $ 
in $H_4(B)$, see also the proof of \cite[Theorem 14] {hillman}.  Let $\phi:=\phi'_{\theta} \circ \phi'$.



We have $\phi \in \Isom \quadtypecM$.  Recall that we have the following
short exact sequence by Lemma \ref{hat2}
             $$
                \xymatrix{0\ar[r]&H^1(M; \cy2)\ar[r]&
                \Isom^{\langle w_2\rangle} \quadtypecM \ar[r]^(0.55){\widehat{j}}&
                \Isom \quadtypecM \ar[r]&1 \ .}
             $$

Choose $\widehat{f}\in \Isom^{\langle w_2\rangle} \quadtypecM$ such
that $\widehat{j}(\widehat{f})=\phi$.  There exists $(W,
\widehat{F})\in \whtildeM $ which maps to $\widehat{f}$, i.e.,
$\widehat{F}\colon W\to \Bw$ and $F|_{\bd_2W}=\widehat{f}$.

Comparison of Wall's\cite{wall} surgery program with Kreck's modified surgery program gives a 
commutative diagram of exact sequences (see \cite{hk2}, Lemma $4.~ 1$)
          $$
             \xymatrix{&\tilde L_6(\bZ[\pi])\ar@{=}[r] \ar[d]&\tilde L_6(\bZ[\pi])\ar[d]&\cr & 
             \cS(M\times I,\bd)\ar[r]\ar[d]&\hM \ar[r]\ar[d]&\hept{M}\cr &
             \cT(M\times I, \bd)\ar[r]\ar[d]&\whtildeM \ar@{>>}[r] \ar[d]& 
             \Isom \quadtypecMw &\cr & L_5(\bZ[\pi]) \ar@{=}[r]&L_5(\bZ[\pi])&\cr}
          $$
The group $\hM$ consists of oriented $h$-cobordisms $W^5$ from $M$ to $M$, under the equivalence
relation induced by $h$-cobordism relative to the boundary.  The
tangential structures $\cT(M\times I, \bd)$, is the set of degree $1$ normal maps $F\colon (W,\bd W) \to
(M\times I, \bd)$, inducing the identity on the boundary.  The
group structure on $\cT(M\times I, \bd)$ is defined as for
$\whtildeM$.  The map $\cT(M\times I, \bd) \to \whtildeM$ takes
$F\colon (W,\bd W) \to (M\times I, \bd)$ to $(W, \widehat{F})\in
\whtildeM$, where $\widehat{F}= \widehat{p_1}\circ F$ (see \cite{wall} for further details).  Let
$\sigma_5\in L_5(\bZ[\pi])$ be the image of $(W, \widehat{F})$.
We further assume that
           \begin{itemize}
              \item[(A3)] The map $\cT(M\times I, \bd)\to L_5(\bZ[\pi])$ is onto.
           \end{itemize}
Let $(W', F')\in \cT(M\times I, \bd)$ map to $\sigma_5$ and let $(W',\widehat{F'})\in \whtildeM $ be the image of $(W', F')$.
Consider the difference of these elements in $\whtildeM$,
                    $$
                        (W'',\widehat{F''}):=(W',\widehat{F'}) \bullet (-W, \hat f^{-1}\bullet\widehat F) \in \whtildeM .
                    $$
Note that $\hat f^{-1}= \widehat{id}_M\colon M \to \Mw$ denotes the map defined by the pair 
$(id_M\colon M \to M, \nu_M\colon M \to BSO)$.
The element $(W'', \widehat{F''})\in \whtildeM $  maps to $0\in L_5(\bZ[\pi_1])$. 
By the exactness of the right-hand vertical sequence there exists an
$h$-cobordism $T$ of $M$ which maps to $(W'', \widehat{F''})$. Let $f$ denote the induced homotopy self equivalence of $M$.  By
construction we have $c\circ f \simeq \phi \circ c$ where $c\circ f= j\circ \widehat{f}$.  Note that $\pi_2(\zeta^{-1}\circ f)=\psi
\oplus \id$ \ and also $\zeta^{-1}\circ f$ \  gives us a self-equivalence of $M_{3/2}$.  Now, if we put the $s$-cobordism $T$
in between the two halves of $W$, then the $3$-handles from the upper half cancel the $2$-handles from the lower half.  
This finishes the proof of Theorem \ref{main}. \hskip 4.8cm $\square$

\bigskip

\noindent{\bf {Acknowledgements.}} This research was supported by the Slovenian-Turkish 
grant BI-TR/12-15-001 and 111T667.  The second author would like to thank Jonathan Hillman for 
very useful conversations.

\bibliographystyle{amsplain}
\providecommand{\bysame}{\leavevmode\hbox
to3em{\hrulefill}\thinspace}
\providecommand{\MR}{\relax\ifhmode\unskip\space\fi MR }
\providecommand{\MRhref}[2]{%
  \href{http://www.ams.org/mathscinet-getitem?mr=#1}{#2}
} \providecommand{\href}[2]{#2}

\end{document}